\documentclass[a4paper,12pt]{article}

\usepackage{amsmath}
\usepackage{amssymb}
\usepackage{amsthm}

\usepackage{algorithm}
\usepackage[noend]{algpseudocode}

\usepackage{enumitem}

\usepackage{hyperref}

\newtheorem{thm}{Theorem}[section]

\newtheorem{lem}[thm]{Lemma}

\newcommand{\Aut}[1]{\mathrm{Aut}({#1})}

\begin{document}

\title{The chromatic number of the square of the $8$-cube}

\author{%
Janne~I.~Kokkala\footnote{Supported by Aalto ELEC Doctoral School, Nokia Foundation, 
Emil Aaltonen Foundation, and by Academy of Finland Project 289002.}\ \ 
and Patric~R.~J.~\"{O}sterg{\aa}rd\footnote{Supported in part by Academy of Finland Project 289002.}\\
\hspace*{5mm}\\
Department of Communications and Networking\\
Aalto University School of Electrical Engineering\\
P.O.\ Box 13000, 00076 Aalto, Finland
}
\date{}
\maketitle

\begin{abstract}
A cube-like graph is a Cayley graph for the elementary abelian group of order $2^n$.
In studies of the chromatic number of cube-like graphs, 
the $k$th power of the $n$-dimensional hypercube, $Q_n^k$, is frequently
considered. This coloring problem can be considered in the framework of
coding theory, as the graph $Q_n^k$ can be constructed with one vertex 
for each binary word of length $n$ and edges between vertices exactly when the 
Hamming distance between the corresponding words is at most $k$. 
Consequently, a proper coloring 
of $Q_n^k$ corresponds to a partition of the $n$-dimensional binary Hamming space 
into codes with minimum distance at least $k+1$. The smallest open case,
the chromatic number of $Q_8^2$, is here settled by finding a 13-coloring. 
Such 13-colorings with specific symmetries are further classified.
\end{abstract}

\section{Introduction}

A \emph{cube-like graph} is a Cayley graph for the elementary abelian group of order $2^n$. One of the
original motivations for studying cube-like graphs was the fact that they have 
only integer eigenvalues \cite{H75}. Cube-like graphs also form a generalization of the hypercube.

There has been a lot of interest in the chromatic number of cube-like graphs
\cite[Sect.\ 9.7]{JT95}. In the early studies, people realized that many types of such
graphs have a chromatic number that is a power of two \cite{DHLL90}. This observation
inspired work into one of the two main research directions that 
have emerged: Determine the spectrum of chromatic numbers of cube-like graphs.
Payan \cite{P92} showed that there are gaps in the spectrum by proving that 3 is 
not a possible chromatic number; he also found a cube-like graph with chromatic number 7,
disproving earlier conjectures that the chromatic number might always be 
a power of two. 

The other main research direction is that of determining the chromatic number for
specific families of cube-like graphs. The $n$-dimensional hypercube, also called the 
$n$-cube and denoted by $Q_n$, is the graph with one vertex for each binary word of
length $n$ and with an edge between two vertices 
exactly when the Hamming distance between the corresponding words is $1$. The $k$th power of 
a graph $\Gamma = (V,E)$ is the graph $\Gamma^k = (V',E')$, where $V'=V$ and in which two vertices
are adjacent exactly when their distance in $\Gamma$ is at most $k$. In the current work, 
we focus on the chromatic number of (the cube-like graph) $Q_n^k$, 
denoted by $\chi_{\bar{k}}(n)$. The chromatic 
number $\chi_{\bar{k}}(n)$ has been studied, for example, in \cite{KDP00,Z01,NDG02,O04} and 
is further motivated by the problem of scalability of certain optical networks \cite{W97}.

The value of $\chi_{\bar{2}}(n)$ has been determined for $n \leq 7$, and for $n = 8$ it has been
known that $13 \leq \chi_{\bar{2}}(8) \leq 14$, where the upper bound follows from 
\mbox{14-colorings} found independently by Hougardy \cite{Z01} and Royle 
\mbox{\cite[Section~9.7]{JT95}}. By finding a 13-coloring of $Q_8^2$, we shall here 
prove that $\chi_{\bar{2}}(8) = 13$. The result is obtained by computer search, where the search
space is reduced by prescribing symmetries. An exhaustive classification is further carried out
in the reduced search space. We also show that none of the colorings of $Q_8^2$ that were found
can occur as a subgraph in a $13$-coloring of $Q_9^2$.

The remainder of the paper is structured as follows. In Section~\ref{sec:definitions}, 
we review the relation between binary codes and the graph $Q_n^k$, give definitions,
and survey some old results. Properties of a putative $13$-coloring of $Q_8^2$ that are 
utilized in the computer search are discussed in Section~\ref{sec:8cube}. The method for 
computational classification is explained in Section~\ref{sec:algorithm}, the results 
are reported in Section~\ref{sec:results}, and a consistency check of the computational 
results is discussed in Section~\ref{sec:doublecount}. Finally, Section~\ref{sec:extend}
describes a method for searching for $13$-colorings of $Q_9^2$ starting from the available
$13$-colorings of $Q_8^2$.

\section{Binary codes and powers of the $n$-cube} 
\label{sec:definitions}

We have seen that the graphs $Q_n$ and $Q_n^k$ are conveniently defined via the
corresponding Hamming space. Similarly, the problem of studying the chromatic number
of these graphs benefits from a coding-theoretic framework. 

A \emph{binary code} of \emph{length} $n$ and \emph{size} $M$ is a subset of $\mathbb{Z}_2^n$
of size $M$. Since all codes in this work are binary, we frequently omit that term and just 
talk about codes. The elements of a code are called \emph{codewords}, and the 
\emph{minimum distance} of a code is the smallest Hamming distance between
any two distinct codewords. A binary code with length $n$, size $M$, 
and minimum distance at least $d$ is called an $(n,M,d)$ code. 
A binary code is called \emph{even} if the Hamming weight of all codewords is even. We denote the set of all even-weight binary words of length $n$ by $\mathbb{E}^n$ and the set of all odd-weight binary words of 
length $n$ by $\mathbb{O}^n$. Determining $A(n,d)$, the largest possible size 
of a code with given $n$ and $d$, is one of the main problems in combinatorial coding theory.

A proper coloring of $Q_n^k$ corresponds to a partition of $\mathbb{Z}_2^n$ into binary 
codes with minimum distance at least $k+1$. The maximum size 
of a color class is $A(n,k+1)$, which implies the lower bound 
\begin{equation}
\label{eq:lower}
\chi_{\bar{k}}(n) \geq \lceil 2^n / A(n,k+1) \rceil.
\end{equation}

For colorings of $Q_n^2$, general constructions \cite{W97,LMT88} give 
\begin{equation}
\label{eq:upper}
\chi_{\bar{2}}(n) \leq 2^{\lceil \log_2 (n+1) \rceil}.
\end{equation}
When $n=2^t-j$ where $j=1,2,3,4$, we have $A(n,3)=2^{2^t-t-j}$ because the $j-1$ times shortened 
Hamming code is optimal \cite{BB77}, so \eqref{eq:lower} and \eqref{eq:upper} coincide. 
With $n \leq 15$, $\chi_{\bar{2}}(n)$ is unknown for $n=8,9,10,11$. For 
these values of $n$, the values of $A(n,3)$ are $20$, $40$, $72$, and $144$, 
respectively \cite{BBMOS78,B80,OBK99}, and \eqref{eq:lower} yields the lower bounds 
$13$, $13$, $15$, and $15$, respectively.
The upper bound $\chi_{\bar{2}}(8) \leq 14$ follows from
$14$-colorings of $Q_8^2$ found independently by Hougardy in 1991 \cite{Z01} and Royle 
in 1993 \cite[Section~9.7]{JT95}. In this work, we shall show that 
$\chi_{\bar{2}}(8) = 13$. Recently, Lauri \cite{L16} reported a 
$14$-coloring of $Q_9^2$, which implies that $\chi_{\bar{2}}(9) \leq 14$. 

The concept of symmetry is essential for the results of the current study. 
Two binary codes are called \emph{equivalent} if one can be obtained from the other by 
a permutation of coordinates and addition of a word in $\mathbb{Z}_2^n$ to each codeword. 
The operations maintaining equivalence of binary codes are the isometries of the 
Hamming space $\mathbb{Z}_2^n$. For even-weight binary codes of length $n$, we require that 
the addition be carried out with even-weight words and denote the group of operations 
maintaining equivalence by $G_n$. 

The halved $n$-cube, $\frac{1}{2}Q_n$, is the graph over the words of $\mathbb{E}^n$ 
that has edges between any two vertices whose Hamming distance is $2$. 
It is well known that $Q_n^2$ is isomorphic to $\frac{1}{2} Q_{n+1}$: adding a parity bit to 
each word in $\mathbb{Z}_2^n$ gives an isomorphism. For $n \geq 4$, the automorphism group 
of $\frac{1}{2} Q_{n+1}$ has order $(n+1)! 2^n$. The automorphisms are precisely the operations 
maintaining equivalence of even binary codes. Further, an independent set in 
$\frac{1}{2} Q_{n+1}$ corresponds to an even binary code of length $n+1$ with minimum distance 
at least $4$. Therefore, it is convenient to use even binary codes of length $n+1$ when 
discussing colorings of the square of the $n$-cube. A proper coloring of $Q_{n}^2$ thus 
corresponds to a partition of $\mathbb{E}^{n+1}$ into even binary codes of minimum distance 
at least $4$. We call partitions of $\mathbb{E}^{n+1}$ and partitions of a 
subset of $\mathbb{E}^{n+1}$ that contain only codes with minimum distance 
at least $4$ \emph{admissible}.

For an element $g\in G_n$ and a codeword $c \in \mathbb{E}^n$, we use the notation $g c$ for 
$g$ acting on $c$. Further, for a code $C \subseteq \mathbb{E}^n$, we denote 
$g C = \{g c : c \in C\}$, and for a set of codes 
$\mathcal{C} \subseteq \mathcal{P}(\mathbb{E}^n)$ we denote 
$g\mathcal{C} = \{g C : C \in \mathcal{C}\}$. Two codes, $C$ and $D$, are equivalent if 
$C=g D$ for some $g \in G_n$. The automorphism group of a code $C$ is the group 
$\Aut{C} = \{g : g C=C\}$. The \emph{orbit} of a code $C \subseteq \mathbb{E}^n$ under a 
group $H \leq G_n$ is the set $\{h C : h \in H\}$. We call two partitions $\mathcal{C}$ 
and $\mathcal{C}'$ of $\mathbb{E}^n$ \emph{isomorphic} if $g\mathcal{C}=\mathcal{C}'$ for 
some $g\in G_n$. The automorphism group of a partition $\mathcal{C}$ is the group 
$\Aut{\mathcal{C}} = \{g : g\mathcal{C} = \mathcal{C}$\}.

The \emph{Hamming code} of length $7$ and size $16$ is the unique $(7,16,3)$ code 
(up to equivalence) code that is a subspace of the vector space $\mathbb{F}_2^n$. The 
extended Hamming code is the even $(8,16,4)$ code obtained by adding a parity bit to each 
codeword of the Hamming code. Adding another parity bit to each codeword ($0$ for all codewords) 
gives the doubly extended Hamming code, which is an even $(9,16,4)$ code.

\section{Partitions of $\mathbb{E}^9$} 
\label{sec:8cube}

As discussed above, a proper coloring of the square of the $8$-cube, $Q_8^2$, corresponds to 
a partition of $\mathbb{E}^9$ into even codes with minimum distance at least $4$.
Let us now consider the distribution of code sizes in such a partition containing
13 codes. As the maximum size of a code is $A(8,3)=A(9,4)=20$, 
there are five different distributions:
\begin{itemize}
\item one code of size $16$, twelve of size $20$,
\item one code of size $17$, one of size $19$, eleven of size $20$,
\item two codes of size $18$, eleven of size $20$,
\item one code of size $18$, two of size $19$, ten of size $20$,
\item four codes of size $19$, nine of size $20$.
\end{itemize}

All attempts by the authors to exhaustively search for partitions of these types
failed, like in (unpublished) earlier studies. See also \cite{R08}.
The authors then decided to restrict the search to partitions with prescribed
automorphism groups, which turned out to be successful as we shall see.

Specifically, we search for, and classify up to isomorphism,
all admissible partitions $\mathcal{C}$ of $\mathbb{E}^9$ for which $|\Aut{\mathcal{C}}| \geq 3$. 
The case when one code is the doubly extended Hamming code of size $16$ leads
to many admissible partitions, and this case is also considered for $|\Aut{\mathcal{C}}| = 2$.

We shall next prove some results that provide the theoretical framework for our search.
Prescribing automorphism groups in the construction of combinatorial objects is a standard 
technique \cite[Section~9]{KO06}, but there are a few more details to take into account when considering 
sets of objects rather than single objects. For example, an automorphism of such a
set may map an object (here, code) onto itself or onto another object.

Theorem~\ref{thm:subgroup} below shows that for the problem at hand, 
all admissible partitions can be found by first fixing 
a code $C$ and a subgroup $H \leq \Aut{C}$ and then searching for all admissible partitions 
$\mathcal{C}$ for which $C \in \mathcal{C}$ and $H \leq \Aut{\mathcal{C}}$. 

\begin{lem} \label{lem:coprime}
Let $\mathcal{C}$ be a partition of\/ $\mathbb{E}^n$ that contains $N$ codes of size $M$. 
Let $H\leq \Aut{\mathcal{C}}$ such that $|H|$ is a power of a prime $p$. If $|H|$ and $N$ are 
coprime, then $H \leq \Aut{C}$ for a code $C$ of size $M$.
\end{lem}
\begin{proof}
Because $h C \in \mathcal{C}$ for each $C \in \mathcal{C}$ and $h \in H$, the set of codes 
of size $M$ is partitioned into orbits by $H$. The size of every orbit must divide $|H|$. 
Because $p$ does not divide $N$, at least one orbit has length $1$. Thus $H$ is a 
group of automorphisms of the single code in that orbit.
\end{proof}

\begin{thm} \label{thm:subgroup}
Let $\mathcal{C}$ be a partition of\/ $\mathbb{E}^9$ into $13$ codes of minimum distance at least $4$.
\begin{enumerate}[label=(\roman*)]
\item If\/ $\mathcal{C}$ contains the doubly extended Hamming code $C$ and $|\Aut{\mathcal{C}}|\geq 2$, 
then there is a subgroup $H \leq \Aut{\mathcal{C}}$ of prime order which is also a subgroup of $\Aut{C}$.
\item If\/ $|\Aut{\mathcal{C}}|\geq 3$, then there is a subgroup $H \leq \Aut{\mathcal{C}}$ whose order is $4$ 
or an odd prime so that $H\leq \Aut{C}$ for some code $C \in \mathcal{C}$.
\end{enumerate}
\end{thm}
\begin{proof} ~

\begin{enumerate}[label=(\roman*)]
\item Because $\Aut{\mathcal{C}}$ is nontrivial, it necessarily contains a subgroup $H$ of 
prime order. Because $C$ is the only code of size $16$ in the partition, $H \leq \Aut{C}$ by 
Lemma~\ref{lem:coprime}. 

\item The group $\Aut{\mathcal{C}}$ has a subgroup $H$ of order $4$ or of order $p$ where $p$ 
is an odd prime. We have the following cases based on the size distribution of $\mathcal{C}$. 
We use Lemma~\ref{lem:coprime} in all cases.
\begin{itemize}
\item $1\times16 + 12\times20$: $H \leq \Aut{C}$ where $|C|=16$.
\item $1\times17 + 1\times19 + 11\times20$: $H \leq \Aut{C}$ where $|C|=17$. 
\item $2\times18 + 11\times20$: Because $H$ is a subgroup of $G_n$, $|H|$ divides $|G_n|=9!2^{8}$. 
Therefore $|H|$ and $11$ are coprime, so there is a code $C$ of size $20$ such that $H \leq \Aut{C}$.
\item $1\times18 + 2\times19 + 10\times20$: $H \leq \Aut{C}$ where $|C|=18$.
\item $4\times 19 + 9\times 20$: We get two cases:
\begin{itemize}
\item $|H|=4$: Because $9$ and $|H|$ are coprime, $H \leq \Aut{C}$ where $|C|=20$.
\item $|H|$ is an odd prime $p$: Because $4$ and $|H|$ are coprime, $H \leq \Aut{C}$ where $|C|=19$.
\end{itemize}
\end{itemize}
\end{enumerate}
\end{proof}

Using more precise language, every partition $\mathcal{C}$ that we wish to find occurs in 
a triple $(\mathcal{C}, C, H)$ for which $C \in \mathcal{C}$, $H \leq \Aut{\mathcal{C}}$ 
and $H \leq \Aut{C}$ where the sizes of $C$ and $H$ are as in one of the cases in the 
proof of Theorem~\ref{thm:subgroup}. 

Because we are eventually interested only in 
constructing nonisomorphic partitions $\mathcal{C}$, we can reduce the search space by 
the following observations. Theorem~\ref{thm:1} shows that it is enough to consider one 
candidate $C$ from each equivalence class. Theorem~\ref{thm:2} shows that for a given code $C$, 
it is enough to consider only one subgroup from each conjugacy class of subgroups of $\Aut{C}$.
\begin{thm}
\label{thm:1}
Let $\mathcal{C}$ be a partition of\/ $\mathbb{E}^n$, let $C$ be a code in $\mathcal{C}$, and 
let $H$ be a group for which $H \leq \Aut{\mathcal{C}}$ and $H \leq \Aut{C}$. Let $D$ be a 
code equivalent to $C$. Then $\mathcal{C}$ is isomorphic to a partition $\mathcal{D}$ for 
which $D \in \mathcal{D}$ and $H' \leq \Aut{\mathcal{D}}$ and $H' \leq \Aut{D}$ where $H'$
is a conjugate of $H$ in $G_n$.
\end{thm}
\begin{proof}
Let $g$ be an isomorphism for which $D = g C$. Let $\mathcal{D} = g \mathcal{C}$. Now 
$\mathcal{D}$ is a partition isomorphic to $\mathcal{C}$ and contains $D$. Finally, for each 
$h \in H$, we have $g h g^{-1} \mathcal{D} = g h \mathcal{C} = g \mathcal{C} = \mathcal{D}$ 
and $g h g^{-1} D = g h C = g C = D$, so $H'=g H g^{-1}$ has the required properties.
\end{proof}

\begin{thm}
\label{thm:2}
Let $\mathcal{C}$ be a partition of\/ $\mathbb{E}^n$, let $C$ be a code in $\mathcal{C}$, and 
let $H$ be a group for which $H \leq \Aut{\mathcal{C}}$ and $H \leq \Aut{C}$. 
Let $H'$ be a conjugate of $H$ in $\Aut{C}$.
Then $\mathcal{C}$ is isomorphic to a partition $\mathcal{D}$ for which 
$C \in \mathcal{D}$, $H' \leq \Aut{\mathcal{D}}$, and $H' \leq \Aut{C}$.
\end{thm}
\begin{proof}
Let $g\in\Aut{C}$ such that $H'=gHg^{-1}$. 
Let $\mathcal{D} = g \mathcal{C}$. Now $C = g C \in \mathcal{D}$, and, for each $h \in H$, 
we have $g h g^{-1} \mathcal{D} = g h \mathcal{C} = g \mathcal{C} = \mathcal{D}$, so $g H g^{-1} \leq \Aut{\mathcal{D}}$. Therefore $\mathcal{D} = g \mathcal{C}$ satisfies the conditions. 
\end{proof}

\section{Computational classification} \label{sec:computational}

\subsection{Algorithm} \label{sec:algorithm}

Before the main search, the authors classified the even $(9,M,4)$ codes for 
$16 \leq M \leq 20$; the number of equivalence classes is 343566, 41499, 2041, 33, 
and 2, respectively. This classification was carried out and validated with 
software developed for \cite{O11}; some of these codes were classified already 
in \cite{OBK99}. The automorphism groups of the codes can be obtained as a by-product
of this classification or by separately using a standard transformation
to a colored graph \cite{OBK99} (see also \cite[pp.~86--87]{KO06}) which is fed
to the graph isomorphism software \emph{nauty} \cite{MP14}. We use the notation
$\mathcal{C}_M$ for a set of representatives of the equivalence classes of 
even $(9,M,4)$ codes.

The main idea of the search algorithm is to start by fixing a code $C$ in the partition
and a group $H$ that is a subgroup of the automorphism groups of $C$ and the partition.
The other codes in the partition are divided into orbits by $H$, so the search proceeds
by finding possible orbits and combining them into partitions of $\mathbb{E}^n$.

The search algorithm is given as Algorithm~\ref{alg:search} in pseudocode. 
The search is carried out by calling $\textproc{Search}(M, N_1, M_1, N_2, M_2)$
for each of the six possible cases regarding size distributions of codes
and choice of the code size $|C|$ particularized in the proof of Theorem~\ref{thm:subgroup}.
The parameters of the call are
as follows. The value of $M$ is the size of the particularized code $C$ in
the proof of Theorem~\ref{thm:subgroup}. Disregarding $C$, there are
one or two sizes for the remaining codes. Let $N_1$ be the number of codes of 
size $M_1$ and $N_2$ the number of codes of size $M_2$, where $0 \leq N_1 \leq N_2$ 
(so $N_1=0$ if there is only one size of remaining codes; then $M_1$ is 
undefined) and $N_1 + N_2 = 12$. 

\begin{algorithm}[htbp]
\caption{Main search procedure} \label{alg:search}
\begin{algorithmic}
\Function{FindOrbits}{$C$: code, $N$, $M$: integers, $H$: group}
\State $\mathcal{S} \gets \emptyset$
\ForAll{$C' \in \mathcal{C}_{M}$}
\ForAll{$g \in G_n$}
\State $\mathcal{O} \gets\emptyset$
\ForAll{$h \in H$}
\State Insert $hgC'$ into $\mathcal{O}$
\EndFor
\If {($|\mathcal{O}| \leq N$ \textbf{and}
\State elements in $\mathcal{O}$ disjoint from $C$ \textbf{and}
\State elements in $\mathcal{O}$ disjoint from each other)
}
\State Insert $\mathcal{O}$ into $\mathcal{S}$
\EndIf
\EndFor
\EndFor
\Return $\mathcal{S}$
\EndFunction
\Procedure{Search}{$M$, $N_1$, $M_1$, $N_2$, $M_2$: integers}
\ForAll{$C \in \mathcal{C}_M$}
\ForAll{$H \in \textproc{NonconjugateSubgroups}(\Aut{C})$}
\State $\mathcal{S}_1 \gets \textproc{FindOrbits}(C, N_1, M_1, H)$
\State $\mathcal{S}_2 \gets \textproc{FindOrbits}(C, N_2, M_2, H)$
\ForAll{$S_1 \in \textproc{Pack}(\mathbb{E}^9\setminus C, \mathcal{S}_1, N_1)$}
\ForAll{$S_2 \in \textproc{Exact}(\mathbb{E}^9\setminus (C \cup \bigcup_{\mathcal{O} \in S_1} \bigcup_{C' \in \mathcal{O}} C'), \mathcal{S}_2)$}
\State Report $\{C\} \cup \left(\bigcup_{\mathcal{O} \in S_1} \mathcal{O} \right) \cup \left(\bigcup_{\mathcal{O} \in S_2} \mathcal{O} \right)$
\EndFor
\EndFor
\EndFor
\EndFor
\EndProcedure
\end{algorithmic}
\end{algorithm}

The following subroutines are called from the search algorithm. We use the notation
$\mathcal{P}(X)$, where $X$ is a set, for the set of all subsets of $X$. 

\vspace*{5mm}
\noindent
$\textproc{Pack}(X, \mathcal{S}, N)$, where $X$ is a set, 
$\mathcal{S} \subseteq \mathcal{P}(\mathcal{P}(X))$, and $N$ is an integer: Finds all subsets $S$ of 
$\mathcal{S}$ where each element of $X$ appears \emph{at most} once and 
$\sum_{\mathcal{O} \in S} |\mathcal{O}| = N$, and returns the set of all such sets $S$.

\noindent
$\textproc{Exact}(X, \mathcal{S})$, where $X$ is a set and 
$\mathcal{S} \subseteq \mathcal{P}(\mathcal{P}(Y))$ for some $Y \supseteq X$: Finds all subsets $S$ of $\mathcal{S} \cap \mathcal{P}(\mathcal{P}(X))$ so that
each element of $X$ appears \emph{exactly} once in $S$, and returns the set of all such sets $S$.

\noindent
$\textproc{NonconjugateSubgroups}(G)$ returns a set containing one representative from each 
conjugacy class of subgroups of $G$ of order $4$ or odd prime ($2$ or odd prime when $C$ is 
the doubly extended Hamming code).

\vspace*{5mm}
The first routine essentially finds cliques in a graph with vertices for sets of
words and edges whenever the corresponding sets are nonintersecting. In
this work a tailored backtrack algorithm was used due to the large number 
of vertices in the corresponding graph. 
For the last two routines, one may use the \emph{libexact} software \cite{KP08} and any
computer algebra software (actually, the groups are so small that even brute force
search performs well), respectively. 

Let $H \leq \Aut{C}$ for a prescribed code $C$.
To find all admissible partitions of $\mathbb{E}^n \setminus C$ into codes with the 
given size distribution that are divided into orbits by $H$, we search for sets 
$\{\mathcal{O}_1,\mathcal{O}_2,\dots,\mathcal{O}_k\}$ for which each 
$\mathcal{O}_i$ is an orbit of a code under $H$ and $\bigcup_i \mathcal{O}_i$ is an admissible 
partition of $\mathbb{E}^n \setminus C$. The algorithm does this by first finding the orbits 
of codes of size $M_1$ and then finding the orbits of codes of size $M_2$. In the search for
an admissible partition, one needs to make sure that the codes are nonintersecting. 
When searching for the orbits of codes
of size $M_2$, the additional requirement that all words should be included into some code is 
beneficial for the search; compare the difference between the routines $\textproc{Pack}$ and 
$\textproc{Exact}$.
In principle, the calls to $\textproc{Pack}$ and $\textproc{Exact}$ could be combined into
one call to $\textproc{Exact}$ but the current approach saves memory and enables more efficient 
parallelization.

Orbits of codes are produced by the function $\textproc{FindOrbits}$. This function finds
all orbits where the codes are pairwise disjoint and disjoint from $C$. A naive method
is here sufficient, looping over representatives $C'$ of all equivalence classes of codes 
of the given size and over all elements $g\in G_n$ to get codes $g C'$. Codes with nontrivial 
automorphism groups are then reported more than once, but the duplicate orbits can be removed 
afterwards.

Once the entire search is ready, isomorphic partitions are rejected and 
the automorphism group orders are determined for all solutions. One may consider the partitions 
as colorings of the graph $Q_{n-1}^2$ and use \emph{nauty} for those graphs. Handling colorings 
with indistinguishable colors is described in the \emph{nauty} manual.

\subsection{Results} 
\label{sec:results}

The search for admissible partitions with automorphism group order at least $2$ containing 
the doubly extended Hamming code yielded $2266$ nonisomorphic partitions. Out of these, 
$266$ have an automorphism group of order $2$ and the other $2000$ have an automorphism group 
of order $4$. The search required $5650$ days of CPU time on a single core of 
Intel Core i7 870 processor. CPU times reported later are for a single core of that
processor. The computations were carried out in a computer cluster.

For other cases, the numbers of partitions found are shown in Table~\ref{tab:results} along 
with the required CPU time, grouped by the size distribution and the size of the initial code 
$C$ in the search. Note that the line corresponding to the distribution $16 + 12 \times 20$ 
does not include the search starting from the doubly extended Hamming code. The two separate 
cases with size distribution $4\times 19 + 9 \times 20$ yielded no common partitions.

\begin{table}
\begin{center}
\begin{tabular}{rrrr}
Size distribution & $|C|$ & \# & CPU time \\
\hline
$16 + 12 \times 20$ & $16$ & $125$ & $128$ days\\
$17 + 19 + 11 \times 20$ & $17$ & $0$ & $162$ days \\
$2 \times 18 + 11 \times 20$ & $20$ & $5$ & $291$ hours\\
$18 + 2 \times 19 + 10 \times 20$ & $18$ & $0$ & $66$ hours\\
$4 \times 19 + 9 \times 20$ & $19$ & $1$ & $42$ days\\
$4 \times 19 + 9 \times 20$ & $20$ & $5$ & $32$ hours\\
\end{tabular}
\caption{Number of partitions}
\label{tab:results}
\end{center}
\end{table}

The number of partitions with each automorphism group order at least $3$ are listed 
in Table~\ref{tab:automs}. In addition, there are $266$ partitions that have automorphism 
group of order $2$ where one code is the doubly extended Hamming code.

\begin{table}
\begin{center}
\begin{tabular}{rr}
$|\Aut{\mathcal{C}}|$ & $\#$ \\
\hline
3 &      1 \\
4 &  2099 \\
6 &      5 \\
8 &    25 \\
9 &      1 \\
12 &      2 \\
24 &      1 \\
48 &      2
\end{tabular}
\caption{Automorphism group orders}
\label{tab:automs}
\end{center}
\end{table}

Two of the partitions found, one with distribution $2 \times 18 + 11 \times 20$ and one 
with distribution $4 \times 19 + 9 \times 20$, contain codes that are not 
maximal. Augmenting these codes yield five new nonisomorphic partitions in 
total, two with trivial automorphism group, which have distributions 
$18 + 2 \times 19 + 10 \times 20$ and $4 \times 19 + 10 \times 20$, and three with 
automorphism group order $2$, one of which have distribution 
$18 + 2 \times 19 + 10 \times 20$ and two of which have distribution $4 \times 19 + 10 \times 20$.

We present here a partition with distribution $16+12\times 20$ that has an 
automorphism group of order $48$. Because all codes of size $20$ in this partition 
lie on the same orbit under the automorphism group, it suffices to 
list an even $(9,16,4)$ code $C_0$, an even $(9,20,4)$ code $C_1$,
and two isomorphisms $g_1$, $g_2$ that generate the automorphism group.
An isomorphism $g$ is given as a pair $(\pi,c)$, where $\pi$ 
is a permutation of $\{1,2,\dots,9\}$ and $c$ is a word in $\mathbb{E}^9$ such that 
$g$ maps each word $c' \in \mathbb{E}^9$ to a word that has $c'_{\pi^{-1}(i)} \oplus c_i$ 
at the $i$th coordinate for each $i$.
\begin{align*}
C_0 = \{&000000000,
      000011011,
      100100101,
      100111110,
      101001010, 
      101010111,\\
      &001101001,
      001110100,
      110010001,
      010011100,
      110100010,
      010101111, \\
      &011000110,
      111001101,
      011110011,
      111111000\}, \\
C_1 = \{&000000011,
      100001101,
      100011010,
      100110100,
      000111001,
      101000110,\\
      &001010101,
      101101000,
      001101111,
      101110011,
      010010000,
      110010111,\\
      &010100101,
      110101011,
      010111110,
      111000001,
      011001100,
      011011011,\\
      &011100010,
      111111101\},\\
g_1 = (&(23)(47)(68),100100101), \\
g_2 = (&(1857)(29)(46),000011011).
\end{align*}
    
This result gives an infinite family of colorings of $Q_n^2$.

\begin{thm}
\label{thm:infinite}
$\chi_{\bar{2}}(9 \cdot 2^{i}-1) \leq 13\cdot 2^i$ for $i \geq 0$. 
\end{thm}
\begin{proof}
The result follows from $\chi_{\bar{2}}(8) = 13$ and
the bound $\chi_{\bar{2}}(2n+1) \leq \chi_{\bar{2}}(n)$ \mbox{\cite[Theorem~1]{O04}.}
\end{proof}

For example, Theorem~\ref{thm:infinite} gives that $\chi_{\bar{2}}(17) \leq 26$, 
but we are not able to determine the exact chromatic number in that 
case. By $5632 \leq A(17,3) \leq 6552$ \cite{LO16,BBMOS78} and \eqref{eq:lower}, we
know that $\chi_{\bar{2}}(17) \geq 21$, and finding better bounds for $A(17,3)$
would not be able to improve the bound given by \eqref{eq:lower} beyond 24. 

\subsection{Double counting} 
\label{sec:doublecount}

To increase confidence in the computational results, we perform a consistency check by 
double counting. The counting is done separately for every size distribution of 
$\mathcal{C}$ and size of the code $C$ listed in the proof of Theorem~\ref{thm:subgroup}. 
We find the number of triples $(\mathcal{C}, C, H)$ where $\mathcal{C}$ is an admissible 
partition of $\mathbb{E}^n$ with the given size distribution, $C$ is a code in $\mathcal{C}$ 
of the given size, $H$ is a group that is a subgroup of $\Aut{\mathcal{C}}$ and $\Aut{C}$, 
and $|H|$ is one of the possibilities in the proof of Theorem~\ref{thm:subgroup}. We obtain 
this number in two ways.

The first way is as follows. For each partition $\mathcal{C}$, let $N(\mathcal{C})$ be the 
number of pairs $(C, H)$ such that $(\mathcal{C}, C, H)$ is a triple to be counted. This can 
be found computationally by looping over all subgroups $H$ of $\Aut{\mathcal{C}}$ of admissible 
order and every code $C$ of the fixed size and checking whether $H \leq \Aut{C}$. Because 
$N(\mathcal{C})=N(\mathcal{D})$ when $\mathcal{C}$ and $\mathcal{D}$ are isomorphic, and 
the number of partitions isomorphic to $\mathcal{C}$ is $|G_n|/|\Aut{\mathcal{C}}|$, the count 
can be obtained by
\[
\sum_\mathcal{C} N(\mathcal{C}) \frac{|G_n|}{|\Aut{\mathcal{C}}|},
\]
where the sum is taken over equivalence class representatives of colorings that are found in the search.

On the other hand, for each pair $(C, H)$ where $C$ is a code of the fixed size and 
$H \leq \Aut{\mathcal{C}}$ is of admissible size, let $N(C,H)$ be the number of colorings for 
which $(\mathcal{C}, C, H)$ is a triple to be counted. This is the number of colorings found 
in the search starting from $C$ and $H$. Because $N(C,H) = N(g C, g H g^{-1})$ for every 
$g \in G_n$, the count can be obtained by looking at only one $C$ from each equivalence class 
of codes and only one $H$ from each conjugacy class of subgroups of $\Aut{C}$. As this is exactly 
what is done in the search, the count can be obtained computationally by
\[
\sum_{C, H} N(C, H)  X(C,H) \frac{|G_n|}{|\Aut{C}|},
\]
where $X(C,H)$ is the number of subgroups of $\Aut{C}$ conjugate to $H$ and the sum is taken over 
all pairs $(C,H)$ for which the search was performed. To this end, the numbers $X(C,H)$ and $N(C,H)$ 
are stored during the search.

\section{Extending colorings}
\label{sec:extend}

In an attempt to find a $13$-coloring of $Q_{9}^2$, one may check whether the classified
$13$-colorings of $Q_8^2$ can occur as a subgraph of such a coloring. Consider an admissible partition 
$\mathcal{C}=\{C_1,C_2,\dots,C_{13}\}$ of $\mathbb{E}^{10}$. Each code $C_i$ can be written as 
$C_i = 0D_i \cup 1E_i$, where $D_i$ is an even-weight code of length $9$ with minimum distance 
$4$ and $E_i$ is an odd-weight code of length $9$ with minimum distance $4$. Now 
$\mathcal{D}=\{D_1,D_2,\dots,D_{13}\}$ is an admissible partition of $\mathbb{E}^{9}$ and 
$\mathcal{E}=\{E_1,E_2,\dots,E_{13}\}$ also corresponds to an admissible partition of 
$\mathbb{E}^9$ if for example the last bit is complemented in each codeword.

Given an admissible partition $\mathcal{D}=\{D_1,D_2,\dots,D_{13}\}$ of $\mathbb{E}^9$, we are to 
determine whether it can be extended to a coloring of $\mathbb{E}^{10}$ in the way described 
above. 
The number of ways to express $256$, the size of $\mathbb{E}^9$, as a sum of $13$ integers 
smaller than or equal to $20$ equals the number of ways to express $13 \times 20 - 256 = 4$ 
as a sum of $13$ nonnegative integers, which is $1820$. Therefore, there are $1820$ 
possible choices for the sizes of the codes in $\mathcal{E}$ when the order matters. 
The algorithm now runs as follows. In Steps 1 and 2 of a complete search, representatives of 
all classified colorings and all possible choices of sizes $M_i$ are considered,
respectively.
\begin{enumerate}
\item Consider an admissible partition 
$\mathcal{D}=\{D_1,D_2,\dots,D_{13}\}$ of $\mathbb{E}^9$.
\item Fix the sizes of the codes in $\mathcal{E}$, denoted by $(M_1,\dots,M_{13})$, 
where $M_i\leq 20$ and $\sum_i M_i = 256$.
\item For all $i=1,\dots,13$, find all possible codes $E_i$ such that $|E_i|=M_i$ and 
$C_i = 0D_i \cup 1E_i$ has minimum distance $4$. 
\item Find a partition of $\mathbb{E}^{10}$ from all possible sets $0D_i \cup 1E_i$.
\end{enumerate}

The task in Step 3 can be formulated in the framework of  clique search: starting from 
$\mathbb{O}^{9}$, we remove the words that have distance less than $3$ to a word in $D_i$ 
(because having such a codeword in $E_i$ would result to a pair of codewords in $C_i$ that 
would have distance less than $4$) and consider the graph over the remaining codewords that 
has an edge between each pair of codewords with Hamming distance at least $4$. Now a possible 
code $E_i$ corresponds to a clique of the size $M_i$ in that graph. To search for cliques in 
the graph, we use the software Cliquer \cite{NO03}.

The task in Step~4 is an instance of the \emph{exact cover problem}: Given a set $X$
and a family $\mathcal{S}$ of subsets of $X$, enumerate all subsets of $\mathcal{S}$ that contain
each element of $X$ exactly once. We use \emph{libexact} to solve the instances.
Actually, it suffices to let 
$X = \{1,2,\ldots ,13\} \cup \mathbb{O}^{9}$ with all possible sets $\{i\} \cup E_i$ in $\mathcal{S}$.

None of the known partitions of $\mathbb{E}^9$ could be extended to an admissible partition 
of $\mathbb{E}^{10}$ containing $13$ codes. The search required 146 hours of CPU time.

\end{document}